\documentclass{article}
\usepackage{hyperref,xurl,cite}

\usepackage[utf8]{inputenc}
\usepackage{amsmath,amsthm,fullpage,graphicx}
\usepackage{hyperref}
\usepackage{listings}
\usepackage{authblk}
\usepackage{xurl}
\newtheorem{theorem}{Theorem}
\newtheorem{lemma}{Lemma}
\newtheorem{corollary}{Corollary}

\newtheorem{conjecture}{Conjecture}

\newcommand{\oeis}[1]{\href{https://oeis.org/#1}{#1}}

\title{Computing the Greatest Common Divisor of Binomial Coefficients $\binom{mn}{mk}$}
\author{Chai Wah Wu}
\affil{IBM Research\\IBM T. J. Watson Research Center, Yorktown Heights, NY, USA\thanks{cwwu@us.ibm.com}}

\date{June 18, 2026\\Latest revision: August 4, 2026}

\begin{document}

\maketitle

\begin{abstract}
The greatest common divisor (GCD) of  $\binom{2n}{2k}$ for $1\leq k<n$ is known to be some power of $2$ times the product of all odd primes $p$ such that $2n=p^i+p^j$. 
We complete the analysis of this GCD by showing that this power of $2$ is either $1$ or $0$ and relates it to Mersenne primes.
We also show how to efficiently compute $\text{GCD}\left\{\binom{mn}{mk}: 1\leq k<n\right\}$ when $n$ and $m$ satisfy certain conditions.
\end{abstract}

\vspace{10px}

\noindent{\textbf{Keywords:}{ number theory, binomial coefficients, greatest common divisor, Mersenne exponent, Mersenne prime, Mih\u{a}ilescu's theorem}} 

\vspace{10px}

\section{Greatest common divisor of binomial coefficients}

The greatest common divisor (GCD) of sets of binomial coefficients has been widely studied over the years \cite{Mendelsohn:1971,Albree1972,Gould1972,Straus1973,Joris1985,Hong2016,Xiao:2022,Chung:2025}, beginning with the study by Ram \cite{ram:1909}. 
It was shown in \cite{ram:1909} that the GCD of $\left\{\binom{n}{k}: 1\leq k<n\right\}$ for $n>1$ is equal to $p$ if $n$ is a power of a prime $p$ and is equal to $1$ otherwise. 

More generally, for $n>1$ let us define $g(m,n)$
as:
$$ g(m,n) = \text{GCD}\left\{\binom{mn}{mk}: 1\leq k<n\right\}$$ 
Thus $g(1,n)=p$ if $n$ is a power of prime $p$ and $g(1,n)=1$ otherwise. It is clear that $g(m,n)$ is a divisor of $\binom{mn}{m}$.

For $p$ prime, let $\text{ord}_p(m)$ denote the $p$-adic valuation of $m$, i.e. the largest integer $k$ such that $p^k$ divides $m$. 
Let $\alpha_m(n)$ denote the digit sum of $n$ written in base $m$.

In \cite{McTague2014} it was shown that $g(2,n)$ is equal to some powers of $2$ times the product of odd primes $p$ such that $2n=p^i+p^j$.
In particular, it was shown that for $p$ an odd prime, 
\begin{equation}
\text{ord}_p(g(2,n)) = \left\{\begin{array}{ll} 1&\text{if } 2n = p^i+p^j\\
												0&\text{otherwise}
							\end{array}
\right. \label{eqn:Q}
\end{equation}
The equation $2n = p^i+p^j$ means that $2n$ is the sum of $2$ powers of $p$. 
For $p=2$, this means that the digit sum of $2n$ in binary representation is either $1$ or $2$, i.e. $1\leq \alpha_2(2n)\leq 2$.
For $p>2$, this means that the digit sum of $2n$ in base $p$ representation is $2$, i.e. $\alpha_p(2n) = 2$.
In \cite{McTague2017} Eq. (\ref{eqn:Q}) was extended to:
\begin{equation*}
\text{ord}_p(g(m,n)) = \left\{\begin{array}{ll} 1&\text{if }\alpha_p(mn)=m\\
												0&\text{otherwise}
							\end{array}
\right.
\end{equation*}
for prime $p$ where $p \equiv 1 \pmod{m}$. 
Note that the case of $m=1$ described above is a special case of this result as all primes $p$ satisfy $p\equiv 1 \pmod{1}$ and $\alpha_p(n) = 1$ if and only if $n$ is a power of $p$.

Recall that Kummer's theorem \cite{kummer:1852} states that $\text{ord}_p(\binom{n}{a})$ is equal to the number of carries when adding $a$ to $n-a$ in prime base $p$.
For $m$ a prime power\footnote{$m$ is a prime power if $m=p^t$ for some prime $p$ and integer $t\geq 0$.}, we have the following result:

\begin{theorem}
If $n>1$ and $m=p^t$ is a prime power for some prime $p$ and $t\geq 0$, then
$\text{ord}_p(g(m,n)) = 1$ if $n$ is a power of $p$ and $\text{ord}_p(g(m,n)) = 0$ otherwise.
\end{theorem}

\begin{proof}
First note that $\text{ord}_p(g(m,n)) = \min_{1\leq k <n} \text{ord}_p(\binom{mn}{mk})$. Next, the number of carries when adding $k$ to $n-k$ in base $p$ is the same as when adding $mk$ to $m(n-k)$. If $n$ is not a power of $p$, then either $n=(r+1)p^j$ for some $0<r<p-1$ or there exists $j$ such that the $j$-th digit of $n$ in base $p$ is $0<r<p$ with $rp^j<n$. Then by setting $k = rp^j < n$, there will be no carries when adding $k$ to $n-k$ and by Kummer's theorem $\text{ord}_p(\binom{mn}{mk}) = 0$ and thus $\text{ord}_p(g(m,n))=0$.
If $n=p^i$, then $i>0$ since $n>1$. Furthermore, $n$ has $i+1$ digits in base $p$. For each $1\leq k<n$, both $k$ and $n-k$ will have less than $i+1$ digits and thus there will be a carry adding $k$ to $n-k$ in base $p$ and thus $\text{ord}_p(\binom{mn}{mk})>0$, i.e $\text{ord}_p(g(m,n))>0$. 
For $k = n/p = p^{i-1}$, 
$n-k = p^{i-1}(p-1)$, $k = p^{i-1}$ and adding $(p-1)p^{i-1}$ to $p^{i-1}$ results in exactly one carry.
As a consequence, Kummer's theorem shows that $\text{ord}_p(\binom{mn}{mk})=1$ and thus $\text{ord}_p(g(m,n))=1$.
\end{proof}

\begin{lemma}
Let $n>1$ and $m<q$ be such that $i$ and $q$ are coprime for all $1\leq i\leq m$. If $\alpha_q(mn)=m$, then $q$ divides  $\binom{mn}{m}$.
\end{lemma}
\begin{proof}
If $\alpha_q(mn)=m$, then $mn = \sum_{j=1}^m q^{i_1}$ for $i_1\leq i_2 \cdots \leq i_m$. Since $n>1$, this implies $i_m>0$. This can be written as
$$ q^{i_1}\left(1+q^{i_2-i_1}\left(1+q^{i_3-i_2} + \cdots \left(1+q^{i_{m-1}-i_{m-2}}(1+q^{i_m-i_{m-1}})\right)\cdots \right)\right)$$
Since $i$ and $q$ are coprime for $1\leq i\leq m$, all we need to show is that $q$ divides $mn-i+1$ for $1\leq i\leq m$.
If $i_1>0$, then $q$ is a factor of $mn$.
If $i_j=0$ for $1\leq j< m$ and $i_{j+1}>0$, then
$mn = (1+1+\cdots+1+q^{i_{j+1}}(1+\cdots)) = j+q^{i_{j+1}}(1+\cdots)$,
i.e. $q$ divides $mn-j>mn-m$ and thus
$q$ divides $\binom{mn}{m}$.
\end{proof}

Combining these results we obtain the following characterization of $g(m,n)$ when $n$ and $m$ satisfy certain conditions:
\begin{theorem} \label{thm:one}
Suppose $n>1$, $m=p^t$ is a prime power for some prime $p$,
and all prime factors $q$ of $\binom{mn}{m}$ is either equal to $p$ or $q\equiv 1\pmod{m}$.
Then
\begin{equation*}
g(m,n) = \left\{\begin{array}{ll} p\ \Pi_{q\in T}\ q & \text{if $n$ is a power of $p$} \\
									   \Pi_{q\in T}\ q&\text{otherwise.}
							\end{array}
\right.
\end{equation*} 
where $T = \left\{q \text{ prime}:q\neq p, q|\binom{mn}{m}, \alpha_q(mn)=m\right\}$.
\end{theorem}

The conditions of Theorem \ref{thm:one} imply that
all prime factors of $\binom{mn}{m}$ is either $p$ or larger than $m$. A consequence of Theorem \ref{thm:one} is that $g(m,n)$ is squarefree for these values of $n$ and $m$.
Table \ref{table:mn} shows values of $m$ and $n$ for which the hypothesis of Theorem \ref{thm:one} is satisfied.

\begin{table}[] \begin{tabular}{|c||c|c|c|c|c|c|c|c|}
\hline $m$ & \multicolumn{8}{|c|}{$n$} \\ \hline \hline
2 & 2 & 3 & 4 & 5 & 6 & 7 & 8 & 9  \\ \hline
3 & 13 & 21 & 61 & 73 & 93 & 133 & 181 & 201  \\ \hline
4 & 73 & 85 & 145 & 337 & 505 & 577 & 733 & 829  \\ \hline
5 & 25 & 1681 & 13105 & 32761 & 53281 & 73441 & 107881 & 120121  \\ \hline
7 & 7393681 & 12043081 & 30401281 & 64144081 & 141641641 & 158911201 & 234067681 &  443298241 \\ \hline
7 & 458925481 & 474856201 & 513893521 & 523144441 & 709925041 & 743422681 &758557801 & 965825281\\ \hline
8 & 1199521 & 17907121 & 401929921 & 456029281 & 511711201 & 566624521 & 915248881 & 1014486481 \\ \hline
\end{tabular}\label{table:mn} 
\caption{Values of $m$ and $n>1$ such that $m = p^t$ is a prime power and all prime factors $q$ of $\binom{mn}{m}$ is either equal to $p$ or $q\equiv 1\pmod{m}$.} \end{table}

As an example of applying Theorem \ref{thm:one}, consider the case $m=3$ and $n=3081$. These values satisfy the conditions of Theorem \ref{thm:one} with $T = \{4621,9241\}$. Since $n$ is not a power of $3$, this implies that $g(3,3081) = 4621\times 9241 = 42702661$. 
For the case $m=4$ and $n=9805$, $T = \{19609, 39217\}$ and $g(4,9805) = 769006153 = 19609 \times 39217$.

Another example is when $m = 5$ and $n = 1118041$ with $$T = \{1397551, 1118041, 2795101, 1863401, 5590201\}.$$
Since $n$ is not a power of $5$, 
\begin{eqnarray*}
g(5,1118041) &=& 45494264542284581850166581839291\\& =& 1397551 \times 1118041\times  2795101\times 1863401\times 5590201
\end{eqnarray*}

Consider $m=5$ and $n = 5^{14} = 6103515625$. In this case 
$T = \{30517578121\}$ and $g(5,5^{14}) = 152587890605 = 5\times 30517578121$.
Similarly, 
\begin{eqnarray*}
g(5,5^{42}) &=& 5\\
g(7,7393681) &=& 197369722557758722635198993396022973319446956674973\\
g(7,12043081) &=& 5068726741059031030172567815491325723\\
g(7,30401281) & = &129005136421383583880464807953850269394414020673\\
g(7,64144081) &=& 5425572087152832027611634438764717957\\
g(7,141641641) &=& 188496545324707039013183201869286575130279554171327\\
g(7,158911201) &=& 2703473987321722596860806988486782327458481\\
g(7,234067681) &=& 14057826305110347970145929851258293244454873\\
g(7,443298241) &=& 177146978708442984644336902411216668588268283174081617\\
g(7,458925481) &=& 1900765546933904074761045273164441260453775833\\
g(7,474856201) &=& 1356432240166023395735850621772056781\\
g(7,513893521) &=& 1860553737921491194870761745173711769\\
g(7,523144441) &=& 1752294286224857420001914372640815035895790514124010764059533907\\
g(7,709925041) &=& 12628199460010050431385564408476474212457106997\\
g(7,743422681) &=& 23644154752880388633014783044596447155948468801515711543\\
g(7,758557801) &=& 26683391382342100791312182399057220031425000659329940971\\
g(7,965825281) &=& 34820514607017234356683965835408929793\\
g(7,4218497641) &=& 53460124075333631944665968335345497375198907272247\\
g(8,1014486481) &=& 226802041625350759102354064704943112794509972621273728769\\
g(8,2836033201) &=& 42896449676470790401
\end{eqnarray*}

\section[GCD of binomial coefficients $(2n,2k)$ for $1\leq k<n$]{GCD of $\binom{2n}{2k}$ for $1\leq k<n$}
For the case $m=2$, all primes $q\neq 2$ are odd and thus $q\equiv 1 \pmod{m}$ and reducing Theorem \ref{thm:one} to this case results in an explicit formula for $g(2,n)$ for all $n>1$:
\begin{theorem} \label{thm:m=2}
For $n>1$, $g(2,n)$ is given by
\begin{equation*}
g(2,n) = \left\{\begin{array}{ll} 2\ \Pi_{q\in T}\ q & \text{if $n$ is a power of $2$} \\
									   \Pi_{q\in T}\ q&\text{otherwise.}
				\end{array}
\right.
\end{equation*} 
where $T$ is the set of odd prime factors $q$ of $n(2n-1)$ such that $2n=q^i+q^j$. 
\end{theorem}

The values of $g(2,n)$ are listed in OEIS \cite{oeis} sequence \oeis{A265388}. 
The values of $n$ such that $g(2,n) =1$ are listed in OEIS sequence \oeis{A265401}. From Theorem \ref{thm:m=2} these are exactly the numbers $n$ such that $n$ is not a power of $2$ and all odd prime factors $q$ of $n(2n-1)$ satisfies $\alpha_q(2n) = 2$.

The squarefree kernel of $n$ (also known as the radical) is defined as the product of all distinct prime factors of $n$ and is denoted as $\mbox{rad}(n)$.

\begin{corollary} \label{cor:divide}
If $n>1$ and $\mbox{rad}(2n-1)$ divides $g(2,n)$ then $2n-1$ is a prime power.
\end{corollary}
\begin{proof}
If $2n-1$ is not a prime power, let $q,r>2$ be two distinct prime factors of $2n-1$. Note that $q$ and $r$ are not prime factors of $n$.
Suppose $2n = q^i+q^j$ with $i\leq j$. Since $q$ is not a prime factor of $n$ and $2n=q^i(q^{j-i}+1)$ this implies $i=0$, i.e. $2n-1=q^j$. Similarly $2n-1=r^l$. This is impossible and thus either $q$ or $r$ is not in $T$ and $\mbox{rad}(2n-1)$ does not divide $g(2,n)$.
\end{proof}

The values of $n$ such that $g(2,n)=n$ are listed in OEIS sequence \oeis{A265402}. 

\begin{corollary} \label{cor:one}
For prime $p$, $g(2,p) = p$ if and only if $2p-1$ is not a prime power.
\end{corollary}
\begin{proof}
Theorem \ref{thm:m=2} implies that $g(2,n)$ is odd if $n$ is not a power of $2$. Since $g(2,2) = 6$, Corollary \ref{cor:one} is true for $p=2$. Let us assume that $p$ is an odd prime. Since $g(2,p)$ is squarefree, this implies that $g(2,p)$ is an odd squarefree number.
Suppose that $2p-1$ is not a prime power.
Then $p\in T$ since $2p=p+p$. Let $q,r>2$ be two distinct prime factors of $2p-1$. Then clearly $q\neq p$ and $r\neq p$.
If $2p = q^i+q^j$ with $i\leq j$, then $2p=q^i(q^{j-i}+1)$ and thus $i=0$, i.e. $2p-1=q^j$ contradicting the fact that $r$ is a prime factor of $2p-1$. 
Therefore $2p\neq q^i+q^j$, $T = \{p\}$ and thus $g(2,p) = p$. This argument also shows that if $2p-1=q^t$ is a prime power, then $q\in T$ and $pq$ divides $g(2,p)$, i.e. $g(2,p)\neq p$. 
\end{proof}

\begin{conjecture} \label{conj:one}
$g(2,n)=n$ if and only if $n$ is prime and $2n-1$ is not a prime power.
\end{conjecture}

The following result gives a necessary condition for numbers $n$ such that $g(2,n) = 2n-1$ (OEIS sequence \oeis{A265403}).
\begin{corollary}
If $n>1$ and $g(2,n)=2n-1$ then $2n-1$ is prime.
\end{corollary}
\begin{proof}
By Corollary \ref{cor:divide} $2n-1$ is a prime power. Since $g(2,n)$ is squarefree, the conclusion follows.
\end{proof}

A consequence of Theorem \ref{thm:m=2} is that $g(2,n)$ is a divisor of the squarefree kernel of $n(2n-1)$. This provides an upper bound of $g(2,n)$.
The numbers $n$ such that $g(2,n)$ achieves this upper bound, i.e. $g(2,n)=\mbox{rad}(n(2n-1))$, 
is listed in OEIS sequence \oeis{A395746}. The following result gives a necessary condition for terms of this sequence
and is a direct consequence of Corollary \ref{cor:divide}

\begin{corollary}
If $n>1$ and $g(2,n)=\mbox{rad}(n(2n-1))$ then $2n-1$ is a prime power.
\end{corollary}

The following result gives a sufficient condition when $g(2,n)=\mbox{rad}(n(2n-1))$.

\begin{corollary} \label{cor:suff}
Suppose $n>1$. If $n$ and $2n-1$  are both prime powers then $g(2,n)=\mbox{rad}(n(2n-1))$.
\end{corollary}
\begin{proof}
Suppose $n=p^i$ and $2n-1=q^j$ are prime powers. Then $q>2$. By Mih\u{a}ilescu's theorem \cite{Mihailescu2004} (also known as Catalan's conjecture), $p=2$ is not possible and thus $p>2$.
$2n = p^i+p^i = q^j+q^0$ so both $p$ and $q$ are in $T$ and thus $g(2,n)= pq =\mbox{rad}(n(2n-1))$.
\end{proof}

We conjecture that the conditions in Corollary \ref{cor:suff} is also necessary except for the special case $n=15$.
\begin{conjecture}
Suppose $n>1$ and $n\neq 15$. Then $g(2,n) = \mbox{rad}(n(2n-1))$ if and only if $n$ and $(2n-1)$ are both prime powers.
\end{conjecture}

\subsection{A characterization of Mersenne primes and exponents in terms of $g(2,n)$}
Recall that if $2^j-1$ is prime, then $j$ is called a Mersenne exponent (which must necessarily be prime) and $2^j-1$ is called a Mersenne prime. It is not known whether there are a finite or an infinite number of Mersenne primes and the search for them is extremely computationally intensive (see \url{https://www.mersenne.org/}). As of this writing only 52 Mersenne primes are known with the largest known Mersenne prime being $2^{136279841}-1$.  The following result gives an interesting characterization of Mersenne primes in terms of $g(2,n)$ and shows that $g(2,2^i)$ is either $2$ or twice a Mersenne prime and all the Mersenne primes can be written in the form $\frac{g(2,2^i)}{2}$.

\begin{theorem}
Assume $i>0$. Then
\begin{eqnarray*}
g(2,2^{i}) \neq 2  &\Leftrightarrow &  i+1 \text{ is a Mersenne exponent} \\ &\Leftrightarrow&   \frac{g(2,2^{i})}{2} = 2^{i+1}-1 \text{ is a Mersenne prime.}
\end{eqnarray*}
\end{theorem}

\begin{proof}
By Theorem \ref{thm:m=2} $g(2,2^{i})\neq 2$ if and only if $2^{i+1}= p^k+p^m$ for some odd prime $p$ and $k\leq m$. Since $p^k+p^m = p^k(p^{m-k}+1)$, $p^k$ divides $2^{i+1}$, and this implies that $k=0$, i.e. $2^{i+1} = p^m+1$. By Mih\u{a}ilescu's theorem 
the only possible solution is when $m\leq 1$. Since $i>0$, $m=0$ is not possible and we are left with $2^{i+1}=p+1$, i.e. $i+1$ is a Mersenne exponent and $p$ is a Mersenne prime. Since $T = \{p\}$, the corresponding value   
of $g(2,2^{i})$ is $2p$ and the proof is complete.
\end{proof}

\section{Conclusions}
We study the GCD of the set $\left\{\binom{mn}{mk}:1\leq k < n\right\}$ and show that under certain conditions this is a squarefree number that can be efficiently computed. We also show a characterization of Mersenne exponents and Mersenne primes in terms of the GCD of $\left\{\binom{2^i}{2k}:1\leq k < 2^{i-1}\right\}$.


\begin{thebibliography}{10}

\bibitem{Mendelsohn:1971}
N.~Mendelsohn, ``Divisors of binomial coefficients,'' {\em The American
  Mathematical Monthly}, vol.~78, pp.~201--202, 1971.

\bibitem{Albree1972}
J.~Albree, ``The gcd of certain binomial coefficients,'' {\em Mathematics
  Magazine}, vol.~45, no.~5, pp.~259--261, 1972.

\bibitem{Gould1972}
H.~W. Gould, ``A new greatest common divisor property of the binomial
  coefficients,'' {\em The Fibonacci Quarterly}, vol.~10, no.~6, pp.~579--584,
  1972.

\bibitem{Straus1973}
E.~G. Straus, ``On the greatest common divisor of some binomial coefficients,''
  {\em The Fibonacci Quarterly}, vol.~11, no.~1, pp.~25--26, 1973.

\bibitem{Joris1985}
H.~Joris, C.~Oestreicher, and J.~Steinig, ``The greatest common divisor of
  certain sets of binomial coefficients,'' {\em Journal of Number Theory},
  vol.~21, no.~1, pp.~101--119, 1985.

\bibitem{Hong2016}
S.~Hong, ``The greatest common divisor of certain binomial coefficients,'' {\em
  Comptes Rendus. Mathématique}, vol.~354, no.~8, pp.~756--761, 2016.

\bibitem{Xiao:2022}
J.~Xiao, P.~Yuan, and X.~Lin, ``The greatest common divisor of certain set of
  binomial coefficients,'' {\em Mathematical Theory and Applications}, vol.~42,
  no.~1, pp.~85--91, 2022.

\bibitem{Chung:2025}
C.-L. Chung, T.-C. Yang, and K.~Zhou, ``The greatest common divisor of sets of
  binomial coefficients with restrictions,'' {\em Contemporary Mathematics},
  pp.~971--985.

\bibitem{ram:1909}
B.~Ram, ``Common factors of $n!/m!(n-m)!$,'' {\em J. Indian Math. Club (Madras)},
  vol.~1, pp.~39--43, 1909.

\bibitem{McTague2014}
C.~McTague, ``The {C}ayley plane and string bordism,'' {\em Geometry \&
  Topology}, vol.~18, no.~4, pp.~2045--2078, 2014.

\bibitem{McTague2017}
C.~McTague, ``On the greatest common divisor of binomial coefficients,'' {\em
  The American Mathematical Monthly}, vol.~124, no.~4, p.~353, 2017.

\bibitem{kummer:1852}
E.~Kummer, ``\"{U}ber die ergänzungss\"{a}tze zu den allgemeinen
  reciprocit\"{a}tsgesetzen,'' {\em Journal f\"{u}r die reine und angewandte
  Mathematik}, vol.~44, pp.~93--146, 1852.

\bibitem{oeis}
{The OEIS Foundation Inc.}, ``The on-line encyclopedia of integer sequences,''
  1996-present.
\newblock Founded in 1964 by N. J. A. Sloane.

\bibitem{Mihailescu2004}
P.~Mih\u{a}ilescu, ``Primary cyclotomic units and a proof of {C}atalans
  conjecture,'' {\em Journal für die reine und angewandte Mathematik (Crelles
  Journal)}, vol.~2004, no.~572, 2004.

\end{thebibliography}
\end{document}